\documentclass[a4paper,11pt,reqno]{amsart} 
\usepackage{amssymb, amsmath, amsthm, mathtools, mathrsfs, bm}
\usepackage{extarrows}
\usepackage[all]{xy}
\usepackage{tikz-cd}
\usepackage{extarrows}
\usepackage{arydshln} 
\usepackage{enumerate}
\usepackage{booktabs,float}
\usepackage{bbm}
\usepackage{xfrac}
\usepackage[square, comma, sort, numbers]{natbib}
\usepackage[colorlinks = true,
            linkcolor = blue,
            urlcolor  = cyan,
            citecolor = black,
            anchorcolor = black]{hyperref}

\theoremstyle{plain}
\newtheorem{theorem}{Theorem}[section]

\newtheorem{lemma}[theorem]{Lemma}

\theoremstyle{definition}

\theoremstyle{remark}
\newtheorem{remark}[theorem]{Remark}

\theoremstyle{clame}
\newtheorem{clame}[theorem]{Clame}

\newcommand{\inclu}[1]{\ensuremath{\stackrel{#1}\hookrightarrow}}
\newcommand{\lr}[1]{\ensuremath{\left \langle #1  \right \rangle }}
\newcommand{\fra}[2]{\ensuremath{\frac{#1}{\left \langle #2  \right \rangle }}}
\newcommand{\m}[2]{\ensuremath{M_{2^{#1}}^{#2}}}

\newcommand{\Z}{\ensuremath{\mathbb{Z}}}

\newcommand{\Cse}{\ensuremath{C_{r}^{n+2,s}}}
\newcommand{\F}[1]{\ensuremath{F^{{#1},s}_{r}}}
\newcommand{\Ce}[1]{\ensuremath{C_{r}^{{#1},s}}}

\title[the $n+3$, $n+4$ dimensional homotopy groups of $\mathbf{A}_n^2$-complexes]{the $n+3$, $n+4$ dimensional homotopy groups of $\mathbf{A}_n^2$-complexes}

\author[T. Jin]{Tian Jin}
\address{College of Mathematics and Physics, Wenzhou University, Wenzhou, Zhejiang \rm{325035}, China}
\email{lesuishenzhizhao@aliyun.com}

\author[Z. Zhu]{Zhongjian Zhu}
\address{College of Mathematics and Physics, Wenzhou University, Wenzhou, Zhejiang \rm{325035}, China}
\email{zhuzhongjian@amss.ac.cn}

\subjclass[2020]{Primary 55P15}
\keywords{Homotopy groups, $\mathbf{A}_{n}^2$-complexes, fibration sequence, relative James construction}

\begin{document}

\begin{abstract}
In this paper, we calculate the $n+3$, $n+4$ dimensional   homotopy groups of indecomposable $\mathbf{A}_n^2$-complexes after localization at 2. 
 This job is seen as  a sequel to P.J.  Hilton's  work on the $n+1,n+2$ dimensional  homotopy groups of $\mathbf{A}_n^2$-complexes (1950-1951).  
 The main technique used  is analysing the homotopy property of $J(X,A)$, defined by B. Gray for a CW-pair $(X,A)$,  which is homotopy equivalent to the homotopy fibre of the pinch map $X\cup CA\rightarrow \Sigma A$. By the way, the results of these homotopy groups  have been used to  make progress on recent popular topic about the homotopy decomposition of the (multiple) suspension of oriented closed manifolds.
\end{abstract}

\maketitle

\tableofcontents

\section{Introduction}

\label{intro}
 Let $\mathbf{A}_{n}^k$ be the  homotopy category consisting of $(n-1)$-connected finite CW-complexes with dimension less than or equal to $n+k$ $(n\geq k+1)$.  The objects of  $\mathbf{A}_{n}^k$ are also called $\mathbf{A}_{n}^k$-complexes.
 
 In 1950, S.C.Chang classified the indecomposable homotopy types in $\mathbf{A}_{n}^2 (n\geq 3)$ \cite{RefChang}, that is
 \begin{itemize}
 	\item [(i)] Spheres:~~ $S^{n}$, $S^{n+1}$, $S^{n+2}$;
 	\item [(ii)] Elementary Moore spaces:~~ $M_{p^{r}}^n$ , $M_{p^{r}}^{n+1}$  where $p$ is a prime, $r\in \mathbb{Z}^+$ and $M_{p^r}^{k}$ denotes the $\mathbf{k+1}$-dimensional Moore space $M(\Z/p^r, k)$, whose  only nontrivial reduced homology is  $\tilde{H}_{k}(M_{p^r}^{k})=\Z/p^r\Z$;
 	\item [(iii)] Elementary Chang complexes:~~ $C_{\eta}^{n+2}$, $C^{n+2,s}$, $C_{r}^{n+2}$, $C_{r}^{n+2,s}$ ( $r,s\in \mathbb{Z}^+$), which are given by the mapping cones of the maps
 	$\eta_{n}:S^{n+1}\rightarrow S^{n}$,  $f^{n+2,s}=2^sj^{n+1}_1+j^n_2\eta_{n}: S^{n+1}\rightarrow S^{n+1}\vee S^n$, $f^{n+2}_r=(\eta_{n}, 2^{r}\iota_{n}):S^{n+1}\vee S^n \rightarrow  S^n$, $f_r^{n+2,s}=(2^sj^{n+1}_1+j^n_2\eta_{n}, 2^{r}j^n_{2}):S^{n+1}\vee S^n \rightarrow S^{n+1}\vee S^n$ respectively. 	
 \end{itemize}
 where a suspended  finite CW-complex $X$, if $X\simeq X_1\vee X_2$ and both $X_1$ and $X_2$ are not contractible, then $X$ is called decomposable; otherwise $X$ is called indecomposable. $\mathbb{Z}^+$ denotes the set of positive integers; $\iota_n\in \pi_{n}(S^n)$ is the identity map of $S^n$; $\eta_{2}$ is the Hopf map $S^3\rightarrow S^2$ and  $\eta_{n}=\Sigma^{n-2}\eta_2$ for $n\geq 3$; $j^{n+1}_1$, resp. $j^n_2$ is the inclusion of $S^{n+1}$, resp. $S^n$, into $S^{n+1}\vee S^n$.  We recommend \cite{Bau n-1n+3,BauseTF5cels,BauseTF6cels,Drozd2005,PZ23free,PZ5cell,PZ6cell} for the work on the classification of homotopy types of   $\mathbf{A}_{n}^k$-complexes and  \cite{C.Costoya,D.Mendez,ZP,ZLP,ZPhyperbolicity,Resultsfrom} for recent work on the homotopy theory of Chang-complexes.

Recently, the topic of the homotopy decomposition of the
(multiple) suspension space of oriented closed manifolds becomes popular \cite{CuterSO6mfd,Huang5mfd,Huang higmfd,Huang 6mfd,HuangLi 7mfd} after So and Theriault initiated research in this area \cite{SoTheriault4mfd}.  
It has direct  applications for the characterization of important concepts in geometry and physics, such as topological K-theory, gauge groups . It is known from  the work mentioned above that the 2-torsion of the homology groups of the given manifolds usually causes great obstruction to obtain a complete characterization of the homotopy decomposition of the (multiple) suspension.   Li and Zhu \cite{Li4mfd,LiZhu 5mfd,LiZhu 6mfd} make progress on this issue for manifolds whose  homology groups are allowed to have 2-torsion elements.  As can be seen in their work, the results of homotopy groups of indecomposable  $\mathbf{A}_{n}^2$-complexes play the key role to help them overcome the obstruction caused by the 2-torsion homology groups.
So computing more homotopy groups of indecomposable  $\mathbf{A}_{n}^2$-complexes is currently becoming urgent. 

Calculating the unstable homotopy groups of finite CW-complexes is a fundamental and difficult problem in algebraic topology. There are a lot of related work concerned on CW-complexes with the number of cells less than or equal to 2, such as spheres, elementary Moore spaces, projective space and so on. For spheres, Toda deals with these homotopy problems by Toda bracket which is a powerful tool to calculate homotopy groups \cite{Toda}.  For elementary Moore spaces and complex projective spaces, the homotopy exact sequences of a CW-pair $(X,A)$ and the EHP sequence are the main tools \cite{X.G.Liu,Mukai:pi(CPn),Mukai,Mukailifting}. J.Wu calculated the homotpy groups of mod 2 Moore spaces by using the functorial decomposition \cite{WJProjplane} and recently, Yang, et al. calculate the homotopy groups of the suspended quaternionic projective plane in \cite{JXYang} by developing the method of relative James construction given by Gray \cite{Gray}.

However, calculating the  unstable homotopy groups of a CW-complex with the number of cells greater than 2 will be more complicated. Only a few on homotopy groups of all indecomposable  $\mathbf{A}_{n}^2$-complexes by taking them as a whole.  In 1950,  P.J.Hilton calculated the $n+1,n+2$-dim homotopy groups of $\mathbf{A}_{n}^2$-complexes \cite{Hilton1950,Hilton1951, Hiltonbook}. Noting that the method given by Yang et al. \cite{JXYang} provides us more opportunities to calculate the homotopy groups of CW-complexs with cells more than 2, Zhu (the second author of this paper) and Pan calculated the 6 and 7 dimensional unstable
homotopy groups of indecomposable   $\mathbf{A}_{3}^2$-complexes under 2-localization \cite{Resultsfrom}.

In this paper we complete this results for  the $n+3$, $n+4$ dimensional homotopy groups of indecomposable  $\mathbf{A}_n^2$-complexes for all $n\geq 3$ (They are in stable range for $n=5$ and $6$ respectively by the Freudenthal suspension theorem). We use the same method used in \cite{JXYang} and \cite{Resultsfrom}, i.e., homotopy analysis of  relative James construction is the key core of problem-solving.

\begin{theorem}\label{MainThm} The $n+3,n+4$-homotopy groups of all $2$-local  noncontractible  indecomposable $\mathbf{A}_n^2$-complexes (except spheres) are listed as follows:
	
$$\footnotesize{\begin{tabular}{|c|c|c|c|c|c|c|c|}
		\multicolumn{8}{c}{$n+3$ dimensional homotopy groups of indecomposable $\mathbf{A}_n^2$-complexes}\\
		\multicolumn{8}{c}{}\\
		\hline
		&$M_{2^r}^n$ &$M_{2^r}^{n+1}$ &$C_{\eta}^{n+2}$&$C_{r}^{n+2}$&$C^{n+2,s}$&$C_{r}^{n+2,s}$\\
		\hline
			$n=3$ & $(2)^{2-\!\epsilon_r}\!+\!2^{r+\!1}, \!r\leq\! 2$ &$4,\!r\!=\!1$&$2$&$(2)^{2-\epsilon_r}$&$(2)^2$& $(2)^{3-\epsilon_r}+$ \\
	  & $2+4+2^{r},r\geq 3$ &$(2)^2,r\!\geq \!2$&&$+2^{r+\epsilon_r}$&$+2^s$& $2^{m_r^s}+2^{r+\epsilon_r}$ \\
		\hline
		$n=4$ & $2\!+\!2^{r+1}+2^{m_{r-1}^2}$ &$4,\!r\!=\!1$&$2+\infty$&$(2)^{2-\epsilon_r}$&$(2)^2$& $(2)^{3-\epsilon_r}+$ \\
	&  &$(2)^2,r\!\geq \!2$&&$+2^{r+1}$&$+\infty$& $2^{r+1}$ \\
	\hline
	$n\geq 5$ & $2+2^{m_{r}^3}$ &$4,\!r\!=\!1$&$4$&$2+2^{m_{r}^2}$&$2+4$& $(2)^2+2^{m_{r}^2}$ \\
 $stable$	&  &$(2)^2,r\!\geq \!2$&&&&  \\
		\hline
	\end{tabular}~}$$
$$Table~1$$

$$\footnotesize{\begin{tabular}{|c|c|c|c|c|c|c|c|}
		\multicolumn{8}{c}{$n+4$ dimensional homotopy groups of indecomposable $\mathbf{A}_n^2$-complexes}\\
		\multicolumn{8}{c}{}\\
		\hline
		&$M_{2^r}^n$ &$M_{2^r}^{n+1}$ &$C_{\eta}^{n+2}$&$C_{r}^{n+2}$&$C^{n+2,s}$&$C_{r}^{n+2,s}$\\
		\hline
		$n=3$ & $(2)^{2}+$ &$2\!+\!2^{r+1}$&$\infty$&$4+2^{r+1}$&$2^{m_s^2}$& $4+2^{s+2}+$ \\
		& $(1-\epsilon_r)\cdotp 4$ &$+2^{m_{r-1}^2}$&&&$+2^{s+2}$& $2^{m_{s-\epsilon_r}^2}+2^{m_{r+1}^{s+1}}$ \\
		\hline
		$n=4$ & $(2)^{2-\!\epsilon_r}+2^{m_r^3}$ & $2+2^{m_r^3}$&$2$&$2+2^{m_{r+1}^3}$&$2+2^{m^3_s}$& $2+2^{m_r^{s+1}}+$ \\
		&  &&&&$+2^{s+1}$& $2^{m_{r+1}^3}+2^{m_{s}^3}$ \\
		\hline
		$n=5$ &$2+2^{m_{r}^3}$ &$2+2^{m_{r}^3}$&$2$&$2+2^{m_{r+1}^3}$&$2+2^{m_{s}^3}$& $2+2^{m_{r+1}^3} $ \\
		&  &&&&& $+2^{m_{s}^3} $\\
		\hline
			$n\geq 6$ & $2^{m_{r}^3}$ & $2+2^{m_{r}^3}$&$0$&$2^{m_{r+1}^3}$&$2^{m_{s}^3}$&$2^{m_{r+1}^3}+2^{m_{s}^3}$ \\
	$stable$	&  &&&&&  \\
		\hline
	\end{tabular}~}$$
$$Table~2$$
In the above tables,  an integer $n$ indicates a cyclic group $\Z_n:=\Z/n\Z$ of order
$n$, the symbol $\infty$ the group $\Z_{(2)}$ (the $2$-local integers), the symbol $``+"$ the direct sum of the groups and $(2)^k$ indicates the direct sum of $k$-copies
of  $\Z_2$. $m_a^b:=min\{a,b\}$ indicates the minimal integer of  $a$ and $b$. $\epsilon_1 = 1$ and $\epsilon_r =0$ for $r \geq 2$. Furthermore, we also set set $\epsilon_{\infty}=0$ when $r=\infty$ is allowed in the following text. 
\end{theorem}

The $n+3$ dimensional homotopy groups of indecomposable $\mathbf{A}_n^2$-complexes in Table 1  are used to give the homotopy decomposition of (multiple) suspension of simply-connected 6-manifolds in \cite{LiZhu 6mfd}. We will use the $n+4$ dimensional homotopy groups  in Table 2 to do further research on the homotopy decomposition of the (multiple) suspension of 7 or 8-dimensional manifolds in the future.

\begin{remark}\label{Rem:MainThm}
	The results for $n=3$ in the top rows of Table 1 and Table 2 are given by the second author and Pan in \cite{Resultsfrom}. Thus we only need to calculate the cases for $n\geq 4$ in the above Tables. By the way, for the elementary Moore spaces $M_{2^r}^2$, we also calculate the $\pi_5(M_{2^r}^2)$ and $\pi_6(M_{2^r}^2)$ in \cite{ZhuJin}.
	\end{remark}

\section{Some notations and  lemmas}

In this paper,  all spaces and maps are in the category of pointed CW-complexes and maps (i.e. continuous functions) preserving basepoint. And we always use $*$ and $0$ to denote the basepoints and the constant maps mapping to the basepoints respectively. We denote $A\hookrightarrow X$  as an inclusion map.  For a homomorphism of abelian groups  $f: A\rightarrow B$, denote $Kerf$ and $Coker f$ the kernal and cokernal  of $f$ respectively.

Let  $(X,A)$ be a pair of  spaces with base point $*\in A$, and suppose that $A$ is closed in $X$. In \cite{Gray}, B.Gray constructed a space $(X,A)_{\infty}$ analogous to the James construction, which is denoted by us as $J(X,A)$ to parallel with the the absolute James construction $J(X)$. In fact, $J(X,A)$ is the subspace of $J(X)$ of words for which letters after the first are in $A$.  Especially,  $J(X,X)=J(X)$. We denote the  $r$-th filtration of $J(X,A)$ by  $J_{n}(X,A):=J(X,A)\cap J_{r}(X)$, which is denoted by Gray as  $(X, A)_r$ in  \cite{Gray}.
For example, $J_{1}(X,A)=X$, $J_{2}(X,A)=(X\times A)/((a,\ast)\thicksim (\ast,a))$ for each $a\in A$. Moreover, by the 
 pushout diagram ($\alpha_n$) given in \cite{Gray}, 
 $J_{n}(X,A)/J_{n-1}(X,A)$ is naturally homeomorphic to $(X\times A^{n-1})/F=X\wedge A^{\wedge (n-1)}$, where $F\subset X\times  A^{n-1}$ is the ``fat wedge" consisting of those points in which one or more coordinates is the base-point.

It is well known that there is a natural weak homotopy equivalence  $\omega:J(X)\rightarrow \Omega\Sigma X$, which is a homotopy equivalence when $X$ is a finite CW-complex, and satisfies
$\footnotesize{\xymatrix{
		&X\ar@/_0.5pc/[rr]_{\Omega\Sigma}\ar@{^{(}->}[r]& J(X) \ar[r]^{\omega}& \Omega\Sigma X } }$,
where $X\xrightarrow{\Omega\Sigma} \Omega\Sigma X$ is the inclusion $x\mapsto \psi$ where $\psi: S^1\rightarrow S^1\wedge X, t\mapsto  t\wedge x .$

Let $X\xrightarrow{f}Y$ be a map.  We always use $C_f$, $F_f$ and  $M_f$ to denote the maping cone ( or say, cofibre ),  homotopy fibre and  mapping cylinder of $f$, $C_f\xrightarrow{p} \Sigma X$ the pinch map and $\Omega\Sigma X\xrightarrow{\partial}F_p\xrightarrow{\mu_p} C_f\xrightarrow{p}\Sigma X$ the homotopy fibration sequence induced by $p$ respectively. Then we get the relative James construction $J(M_f,X)$  (resp. $r$-th relative James construction $J_{r}(M_f,X)$ ) for the pair $(M_f,X)$.

The following Lemmas \ref{lemma Gray} and  \ref{J(X,A)toJ(X',A')}  are easily  from the theorems of \cite{Gray} for $(M_f,X)$.
\begin{lemma}\label{lemma Gray} Let $X\xrightarrow{f} Y$ be a map.  Then we have
	\begin{itemize}
		\item [(i)]  $F_{p}\simeq J(M_f,X)$;
		\item [(ii)] $\Sigma J(M_f,X)\simeq \bigvee_{k\geq 0}(\Sigma Y\wedge X^{\wedge k})$;$\Sigma J_{k}(M_f,X)\simeq \bigvee_{i= 0}^{k-1}(\Sigma Y\wedge X^{\wedge i})$;
		\item [(iii)] If $Y=\Sigma Y'$, $X=\Sigma X'$, then  $J_2(M_f,X)\simeq Y\cup_{\gamma}C(Y\wedge X')$, where $\gamma=[id_Y, f]$ is the generalized Whitehead product.
	\end{itemize}
\end{lemma}

Denote both the inclusion  $Y\hookrightarrow J_{2}(M_f,X)$ and the composition of the inclusions
$Y\hookrightarrow J_{2}(M_f,X)\hookrightarrow J(M_f,X) \simeq  F_{p}$ by $j_{p}$ without ambiguous. 
Moreover, if $X$ and $Y$ are $m_1-1$, $m_2-1$-connected respectively, then $j_{p}$ is $m_1+m_2-1$ connected map (i.e, it induces isomorphism  on homotopy group with dimension less than $m_1+m_2-1$ and  epimorphism on homotopy group with dimension $m_1+m_2-1$).

\begin{lemma}\label{J(X,A)toJ(X',A')}
	Suppose the left  diagram  is commutative\\
	$ \footnotesize{\xymatrix{
			X\ar[d]^{\mu'}\ar[r]^-{f} & Y\ar[d]^{\mu}\\
			X'\ar[r]^-{f'} & Y'}}$;~~~~~~$ \footnotesize{\xymatrix{
			F_p\simeq  J(M_f,X)\ar[d]^{J(\widehat{\mu},\mu')}\ar[r]&M_f/X\simeq C_f  \ar[d]^{\bar{\mu}}\ar[r]& \Sigma X\ar[d]^{\Sigma\mu'}\\
			F_{p'}\simeq J(M_{f'},X')\ar[r]& M_{f'}/X'\simeq C_{f'} \ar[r]&\Sigma X'}}$
	
	then it induces the right commutative diagrams on fibrations, where $\widehat{\mu}$ satisfies
	
	$ \footnotesize{\xymatrix{
			&Y\ar@/^0.5pc/[rrr]^{\mu}\ar@{^{(}->}[r]_{\simeq}& M_f \ar[r]_{\widehat{\mu}}& M_{f'}\ar[r]_{\simeq}& Y'} }$.  Let
	
	$M_f=J_{1}(M_f,X)\xrightarrow{J(\widehat{\mu},\mu')|_{M_f}=J_{1}(\widehat{\mu},\mu')=\widehat{\mu}} J_{1}(M_{f'},X')=M_{f'}$,
	
	$J_{2}(M_f,X)\xrightarrow{J(\widehat{\mu},\mu')|_{J_{2}(M_f,X)}=J_{2}(\widehat{\mu},\mu')} J_{2}(M_{f'},X')$,
	
	then we have the following commutative diagram
	$$ \footnotesize{\xymatrix{
			Y\wedge X \ar[d]_{\simeq} \ar[r]^{\mu\wedge\mu'} & Y'\wedge X' \ar[d]_{\simeq}\\
			J_{2}(M_f,X)/J_{1}(M_f,X)\ar[r]^-{\overline{J_{2}(\widehat{\mu},\mu')}} & J_{2}(M_{f'},X')/J_{1}(M_{f'},X')} }$$
\end{lemma}

The following Lemma \ref{lem: partial calculate1} comes  from \cite{Resultsfrom},  while the  special case of it for $X=S^m$ comes from \cite{JXYang} in original.
\begin{lemma}\label{lem: partial calculate1}
	Let $X\xrightarrow{f}Y$ be a map. Then the following diagram is homotopy commutative
	$$ \footnotesize{\xymatrix{
			X\ar[d]_{\Omega\Sigma } \ar[r]^{f} & Y \ar@{^{(}->}[d]\\
			\Omega\Sigma X \ar[r]^{\partial} & F_p } }$$
\end{lemma}

\begin{lemma}[Theorem 1.16 of \cite{Cohen}]\label{stable exact seq}
	If $X\xrightarrow{f} Y$ is a map, $X$ is $n-1$ connected, $C_f$ is $m-1$ connected, the dimension of $W$ is less than or equal to $m+n-2$, then we have the exact sequence $$[W,X]\xrightarrow{f_{\ast}} [W,Y] \rightarrow [W,C_{f}].$$
\end{lemma}

\begin{lemma}\label{lem:exact surj iso}
		Let $X\xrightarrow{f} Y \xrightarrow{i}C_f \xrightarrow{p} \Sigma X$ be a cofibration sequence. If the suspension homomorphism $\Sigma :\pi_{n}(X)\rightarrow \pi_{n+1}(\Sigma X)$  is surjective and $j_{p\ast}: \pi_{n}(Y)\rightarrow \pi_{n}(J_{2}(M_f, X))$ is isomorphic, then we have the following exact sequence 
		\begin{align*}
			\pi_{n}(X)\xrightarrow{f_{\ast}}\pi_{n}(Y) \xrightarrow{i_{\ast}}\pi_{n}(C_f).
		\end{align*} 
	\end{lemma}
\begin{proof}
	The lemma is easily obtained by the following commutative diagram, where the top row is exact
		\begin{align*} \small{\xymatrix{
				\pi_{n+1}( \Sigma X)\ar[r]^-{ \partial_{\ast}}&\pi_{n}(F_{p})\ar[r]&\pi_{n}(C_f).\\
				\pi_{n}(X)\ar[r]^-{f_{\ast}}\ar[u]_{\Sigma } &\pi_{n}(Y)\ar[ur]_{i_{\ast}}\ar[u]_{\cong j_{p\ast}} }}  
	\end{align*}
\end{proof}

\begin{lemma}\label{lem:exact m1+m2-1}
	Let $X\xrightarrow{f} Y \xrightarrow{i}C_f \xrightarrow{p} \Sigma X\xrightarrow{-\Sigma f} \Sigma Y $be a cofibration sequence. $X$ and  $ Y$ are $m_1-1$ and $m_2-1$ connected suspended spaces respectively, where $m_1\geq m_2$. Let $n\leq m_1+m_2-1$. Then we have the exact sequence 
	\begin{align}
		\pi_{n+1}(X)\xrightarrow{\partial_{\ast}}\pi_{n}(F_p) \xrightarrow{} \pi_{n}(C_f) \xrightarrow{p_{\ast}} \pi_{n}(\Sigma X) \xrightarrow{ (\Sigma f)_{\ast}}  \pi_{n}(\Sigma Y).  \label{exact Cf1}
	\end{align}
	Moreover if $j_{p\ast}: \pi_{n}(Y)\rightarrow \pi_{n}(J_{2}(M_f, X))$ is isomorphic or especially  $[id_Y, f]=0$, then there is the following exact sequence 
	\begin{align}
		\pi_{n}(X)\xrightarrow{f_{\ast}}\pi_{n}(Y) \xrightarrow{i_{\ast}}\pi_{n}(C_f) \xrightarrow{p_{\ast}} \pi_{n}(\Sigma X) \xrightarrow{ (\Sigma f)_{\ast}}  \pi_{n}(\Sigma Y).  \label{exact Cf}
	\end{align}
\end{lemma}

\begin{proof}
	The sequence (\ref{exact Cf1}) is deduced from the lemma \ref{stable exact seq} (Note that $Ker (\Sigma f)_{\ast}=Ker (-\Sigma f)_{\ast}$, so we replace $(-\Sigma f)_{\ast}$ by  $(\Sigma f)_{\ast}$  in the sequence).  
	
	Since $[id_Y, f]=0$, from Lemma \ref{lemma Gray}, $J_{2}(M_f, X)\simeq Y\vee (Y\wedge X)$, which implies that $j_{p\ast}: \pi_{n}(Y)\rightarrow \pi_{n}(J_{2}(M_f, X))$ is isomorphic for $n\leq m_1+m_2-1$. Moreover, the suspension homomorphism $\Sigma :\pi_{n}(X)\rightarrow \pi_{n+1}(\Sigma X)$  is surjective from the Freudenthal Suspension theorem. Thus the sequence (\ref{exact Cf}) is also exact on the place $\pi_{n}(Y)$ by the Lemma \ref{lem:exact surj iso}. 
\end{proof}

\begin{lemma}\label{lem: Suspen Cf}
	Let $X\xrightarrow{f}Y\stackrel{j_f}\hookrightarrow C_f\xrightarrow{p}\Sigma X$ be a cofibration sequence of CW complexes. Consider the fibration sequence  $\Omega\Sigma^{k+1}X\xrightarrow{\partial^k}F_{\Sigma^kp} \xrightarrow{i_k}\Sigma^k C_f\xrightarrow{\Sigma^kp}\Sigma^{k+1}X$.
Assume that the following restriction of $\Sigma^k: \pi_{n}(\Sigma X)\rightarrow \pi_{n+k}(\Sigma^{k+1}X)$ is isomorphic
	\begin{align}
		 Ker(\partial_{n\!-\!1\ast}\!\!:\!\pi_{n}(\Sigma X)\!\rightarrow \!\pi_{n\!-\!1}(F_p))\xrightarrow{\Sigma^k}  Ker(\partial^k_{n\!+\!k\!-\!1\ast}\!\!:\!\pi_{n+k}(\Sigma X)\!\rightarrow \pi_{n\!+\!k-\!1}(F_{\Sigma^kp})) \label{map:restrictKer}
	\end{align}
If $0\rightarrow Coker \partial_{n\ast}\rightarrow \pi_{n}(C_f)\rightarrow  Ker \partial_{n-1\ast} \rightarrow 0$
	is split, then so is the short exact sequence
\begin{align*}
	0\rightarrow Coker \partial^k_{n+k\ast}\rightarrow \pi_{n}(\Sigma^k C_f)\rightarrow  Ker \partial^k_{n+k-1\ast}\rightarrow 0.
\end{align*}
\end{lemma}
\begin{proof}
	By the proof of Lemma 3.8 in \cite{ZhuJin}, the homomorphism $\Sigma^k$ in (\ref{map:restrictKer}) is well-defined. 
	There is the following homotopy commutative square with $\Omega^k\Sigma^k$ representing the $k$-fold canonical inclusions
	\begin{align*}
		\footnotesize{\xymatrix{
			C_f \ar[r]^-{p}\ar@{^{(}->}[d]_{\Omega^k\Sigma^k } & \Sigma X \ar@{^{(}->}[d]^{\Omega^k\Sigma^k}\\
			\Omega^k\Sigma^k C_f \ar[r]^-{\Omega^k\Sigma^kp} & \Omega^k\Sigma^{k+1} X } }
	\end{align*}
Thus we have the following short exact sequence with a commutative square 
	$$\xymatrix{
& &\pi_{n}(C_f)\ar[r]^-{p_{\ast}} \ar[d]^{\Sigma^k}& Ker \partial_{n\!-\!1\ast}\ar@/^/[l]^-{\sigma} \ar[r]\ar[d]_{\cong \Sigma^k}& 0 \\
	0\ar[r] & Coker \partial^k_{n+k\ast} \ar[r] &\pi_{n+k}(\Sigma^{n+k}C_f)\ar[r]^-{(\Sigma^kp)_{\ast}} & Ker \partial^k_{n+k-1\ast} \ar[r]& 0 .\\
}$$
Since there is a section $\sigma$  of $p_{\ast}$ in the top sequence, the bottom short exact sequence is also split with a section $(\Sigma^k)\sigma (\Sigma^k)^{-1}$ of $(\Sigma^kp)_{\ast}$.
\end{proof}

For an abelian group $G$ generated by $x_1,\cdots,x_n$, denote $G\cong C_1\{x_1\} \oplus\cdots\oplus C_t\{x_t\} $ if $x_i$ is a generator of the cyclic direct summand $C_i$, $i=1,\cdots,t$;  
$\lr{u_1,\cdots, u_l}$ the subgroup generalized by the elements $u_1,\cdots, u_l\in G$.
Specifically, denote $\Z\{x_1,\cdots, x_t\}$ the $\Z\{x_1\}\oplus\cdots\oplus \Z\{x_t\}$. For an element $g\in G$, we abuse the symbol ``$g$" as the equivalent class of $g$ in the quotient group $G/H$ for some subgroup $H$ of $G$.

The following generators of homotopy groups of spheres after localization at 2  come from \cite{Toda}.
$\iota_n\in\pi_n(S^n)$; $\pi_{n+1}(S^n)=\Z_2\{\eta_n\} (n\geq 3)$;
$\pi_{n+2}(S^n)=\Z_2\{\eta^2_n\} (n\geq 3)$; $\pi_{n+3}(S^n)=\Z_8\{\nu_n\} (n\geq 5)$; $\pi_{6}(S^3)=\Z_4\{\nu'\}$;  $\pi_{7}(S^4)=\Z_4\{\Sigma\nu'\}\oplus \Z_{(2)}\{\nu_4\}$; $\pi_{7}(S^3)=\Z_2\{\nu'\eta_6\}$; $\pi_{8}(S^4)=\Z_2\{\Sigma \nu'\eta_7\}\oplus\Z_2\{\nu_4\eta_7\}$; $\pi_{9}(S^5)=\Z_2\{\nu_5\eta_8\}$, where $\eta^2_n=\eta_n\eta_{n+1}$ (similarly, $\eta^3_n=\eta_n\eta_{n+1}\eta_{n+2}$).

We denote  $2^r=0$ when $r=\infty$, i.e, $\Z_{2^{\infty}}=\Z_{0}=\Z_{(2)}$ (after 2-localization).

\section{Elementary Moore spaces}
\label{sec: Moore space}
In this section we calculate homotopy groups of elementary Moore spaces appeared in the Theorem \ref{MainThm}.

 There is a canonical cofibration sequence
\begin{align}
	S^k\xrightarrow{2^{r}\iota_k} S^{k}\xrightarrow{i_{k}} M_{2^r}^{k}\xrightarrow{p_{k}}  S^{k+1} \label{Cofiberation for Mr^k}
\end{align}
We set $M_{2^{\infty}}^{k}=M_{0}^{k}=S^{k}\vee S^{k+1}$ for $r=\infty$.

Let $\Omega S^{k+1}\xrightarrow{\partial}  F_{p_{k}}\rightarrow M_{2^r}^{k}\xrightarrow{p_{k}}  S^{k+1}$ be the homotpy fibration sequence which implies the following   exact squence 
\begin{align}
   \pi_{m+1}(S^{k+1})\xrightarrow{\partial_{m\ast}}\pi_{m}(F_{p_{k}})\rightarrow\pi_{m}(\m{r}{k})\xrightarrow{p_{k\ast}}\pi_{m}(S^{k+1})\xrightarrow{\partial_{m-1\ast}}\pi_{m-1}(F_{p_{k}})
\end{align}

 we get $\Sigma F_{p_{k}}\simeq S^{k+1}\vee S^{2k+1}\vee\dots$ and the $(3k-1)$-skeleton  $Sk_{3k-1}(F_{p_{k}})\simeq S^k\cup_{2^k[\iota_k,\iota_k]}CS^{2k-1}$.

\begin{lemma}\label{pin+3M2r} $\pi_{n+3}( M_{2^r}^{n+1})$,  $\pi_{n+3}( M_{2^r}^{n}) (n\geq 4)$  are listed as 
	\begin{enumerate}[(1)]
		\item \label{pin+3Mn+1} $\pi_{n+3}( M_{2^r}^{n+1})\cong \left\{
		\begin{array}{ll}
		    \Z_4, & \hbox{$r=1$;} \\
			\Z_2 \oplus\Z_2, & \hbox{$r\geq 2$}
		\end{array}
		\right. (n\geq 4)$
		\item \label{pi7M4} $\pi_{7}(M_{2^r}^4) \cong  \Z_{2^{min\{2,r-1\}}}\oplus\Z_{2^{r+1}}\oplus\Z_2$
		\item \label{pin+3Mn} $\pi_{n+3}(M_{2^r}^n)\cong \Z_2\oplus \Z_{2^{min\{3,r\} }} (n\geq 5).$
\end{enumerate}
\end{lemma}
\begin{proof}~~

\textbullet \quad  (\ref{pin+3Mn+1}) and (\ref{pi7M4})  come from Section 2 of \cite{Bau n-1n+3} and Section 3.3 of \cite{Resultsfrom} respectively.
	
\textbullet \quad For (\ref{pin+3Mn}), if $r=1$, $\pi_{n+3}(M_{2}^n)$ is obtained from Lemma 5.2 and Theorem 5.11 of \cite{WJProjplane}.

If $r\geq 2$, $\pi_{n+3}(M_{2^r}^n), \pi_{n+2}(M_{2^r}^n)$ are 
stable for $n\geq 5$, so we have the following  exact  sequences for $k=n+2,n+3$
\begin{align*}
	\pi_{k}(S^n)\xrightarrow{(2^{r}\iota_n)_{\ast}}\pi_{k}(S^n) \xrightarrow{i_{n\ast}}\pi_{k}(M_{2^r}^n) \xrightarrow{p_{n\ast}}\pi_{k}(S^{n+1})\xrightarrow{(-2^{r}\iota_{n+1})_{\ast}}\pi_{k}(S^{n+1})
\end{align*}
So we have the following commutative diagram of short exact sequences which is induced by $\eta: S^{n+3}\rightarrow S^{n+2}$
$$\xymatrix{
	0\ar[r]& \Z_2\ar[r]\ar[d] & \pi_{n+2}(\m{r}{n})\ar[r]^-{p_{n\ast}} \ar[d]^{\eta^{\ast}}& \Z_2\ar@/^/[l]^-{\sigma} \ar[r]\ar@{=}[d]& 0 \\
	0\ar[r] & \Z_{2^{min\{3,r\}}} \ar[r] &\pi_{n+3}(\m{r}{n})\ar[r]^-{p_{n\ast}} & \Z_2 \ar[r]& 0 \\
}$$
It is known that the upper exact sequence splits for $r\geq 2$. If $\sigma$ is the section of $p_{n\ast}$ in the top sequence, then $\eta^{\ast}\sigma$ is the section of the  lower $p_{n\ast}$. Hence the bottom short exact sequence is also split, which implies (\ref{pin+3Mn}) of this lemma.
	
\end{proof}

Note that $\pi_{n+4}( M_{2^r}^{n+1})(n\geq 4)$ in Table 2 and  $\pi_{n+3}( M_{2^r}^{n})(n\geq 5)$ in Table 1 must be the same.  So for the homotopy groups of elementary Moore spaces in Table 2, we only need to calculate $\pi_{n+4}( M_{2^r}^{n}) (n\geq 4)$.
 
\begin{lemma}\label{pin+4M2r}  $\pi_{n+4}( M_{2^r}^{n}) (n\geq 4)$  are listed as 
	\begin{enumerate}[(1)]
		\item \label{pi8M4} $\pi_{8}(M_{2^r}^4)\cong
		\Z_{2^{min\{3,r\}}}\oplus\Z_{2}\oplus(1-\epsilon_r)\Z_2.$
		\item  \label{pi9M5} $\pi_{9}(M_{2^r}^5)\cong \Z_{2^{min\{3,r\}}}\oplus \Z_2.$
		\item \label{pin+4Mn} $\pi_{n+4}(M_{2^r}^n)\cong \Z_{2^{min\{3,r\}}}
		(n\geq 6).$
	\end{enumerate}
\end{lemma}
\begin{proof}~~
	
\textbullet \quad For $M_{2^r}^{4}$, the 9-skeleton  $Sk_{9}(F_{p_{4}})\simeq S^4\bigcup_{\gamma}CS^7$ with $\gamma=2^{r}[\iota_{4},\iota_{4}]=2^{r+1}\nu_4-2^r\Sigma\nu'$, by Lemma \ref{lem: partial calculate1}, we have exact sequence with commutative squares
\begin{align} \small{\xymatrix{
			\pi_{9}( S^5)\ar[r]^-{ \partial_{8\ast}}&\pi_{8}(F_{p_4})\ar[r]&\pi_{8}(M_{2^r}^{4})\ar[r]^-{p_{4\ast}}&\pi_{8}(S^5)\ar[r]^-{ \partial_{7\ast}}& \pi_{7}(F_{p_4})\\
			\pi_{8}(S^4)\ar[r]^-{(2^r\iota_{4})_{\ast}}\ar[u]_{\Sigma } &\pi_{8}(S^4)\ar@{^{(}->}[u]_{ j_{p_4\ast}}&&\pi_{7}( S^4)\ar[r]^-{(2^r\iota_{4})_{\ast}}\ar[u]_{\Sigma } &\pi_{7}(S^4)\ar@{^{(}->}[u]_{j_{p_4\ast}} }}  \label{exact pi8(M^4)}
\end{align}
From the exact sequence of $\pi_{7}(F_{p_4})$ and  (3.16) in \cite{Resultsfrom},  we have
\begin{align}
	&\pi_{7}(F_{p_4})=\left\{
	\begin{array}{ll}
		\!\!\Z_2\{j_{p_4}\Sigma \nu'\!-\!2j_{p_4}\nu_4\} \oplus\Z_8\{j_{p_4}\nu_4\}, & \hbox{$r=1$;} \\
		\!\!\Z_4\{j_{p_4}\Sigma \nu'\} \oplus\Z_{2^{r+1}}\{j_{p_4}\nu_4\}, & \hbox{$r\geq 2$}
	\end{array}
	\right. \label{pi7Fp4}\\
	&\partial_{7\ast}(\nu_5)=2^{2r}j_{p_4}\nu_4-2^{r-1}(2^r-1)j_{p_4}\Sigma\nu'.
	\label{partial7 on nu5}
\end{align}

Thus 
\begin{align*}
Ker \partial_{7\ast}=\Z_{2^{min\{3,r\}}}\{2^{u_r}\nu_5\}, \text{where}~u_1=2, u_2=1, u_r=0 (r\geq 3).
\end{align*}

Consider the confibration sequence 
 $S^7\xrightarrow{\gamma}S^4\stackrel{j_{p_4}}\hookrightarrow Sk_{9}(F_{p_{4}})\xrightarrow{\check{p}_8} S^8$, we have the following fibration sequence 
 \begin{align*}
 	\Omega S^8\xrightarrow{\check{\partial}} F_{\check{p}_8}\rightarrow Sk_{9}(F_{p_{4}})\xrightarrow{\check{p}_8} S^8,  ~\text{where}~ S^4 \inclu{j_{\check{p}_8}} Sk_{9}(F_{\check{p}_8}) ~\text{is 10-connected}.
 \end{align*}
By Lemma \ref{lem: partial calculate1} and Lemma \ref{lem:exact surj iso}, there is the exact squence with commutative square
 \begin{align} \small{\xymatrix{
			\pi_{8}( S^7)\ar[r]^-{ \gamma_{\ast}}&\pi_{8}(S^4)\ar[r]^-{j_{p_{4}\ast}}&\pi_{8}(F_{p_4})\ar[r]^-{\check{p}_{8\ast}}&\pi_{8}(S^8)\ar[r]^-{ \check{\partial}_{7\ast}}& \pi_{7}(F_{\check{p}_{8\ast}})\\
			&&&\pi_{7}( S^7)\ar[r]^-{\gamma_{\ast}}\ar[u]_{\Sigma \cong  } &\pi_{7}(S^4)\ar[u]_{\cong j_{p_4\ast}} }}  \label{exact pi8(Fp4)}
\end{align}
It is easy to see that $\gamma_{\ast}=0$ and $\check{\partial}_{7\ast}$ is injective in (\ref{exact pi8(Fp4)}).
So
\begin{align*}
	\pi_{8}(S^4)\cong \pi_{8}(F_{p_4})=\Z_2\{j_{p_4}\Sigma \nu'\eta_7\} \oplus\Z_2\{j_{p_4}\nu_4\eta_7\}.
\end{align*}

By the left commutative square of (\ref{exact pi8(M^4)}) and Lemma A.1 of \cite{Resultsfrom}, 
\begin{align}
	\partial_{8\ast}(\nu_5\eta_8)&=\partial_{8\ast}(\Sigma\nu_4\eta_7)=j_{p_{4}\ast}[(2^r\iota_4)\nu_4\eta_7]=\left\{
	\begin{array}{ll}
		-j_{p_{4}}\Sigma\nu'\eta_7, & \hbox{$r=1$;} \\
		0, & \hbox{$r\geq 2$}
	\end{array}
	\right.
	\label{partial8 on  nu5}
\end{align}

Thus $Coker\partial_{8\ast}=\left\{
\begin{array}{ll}
	\Z_2\{j_{p_4}\nu_4\eta_7\}, & \hbox{$r=1$;} \\
	\Z_2\{j_{p_4}\Sigma \nu'\eta_7\} \oplus\Z_2\{j_{p_4}\nu_4\eta_7\}, & \hbox{$r\geq 2$}
\end{array}
\right..$

The diagram (3) of \cite{Mukai} induces the following  commutative diagram of short exact sequences for $r>1$ by Lemma \ref{J(X,A)toJ(X',A')}
\begin{align}
	\! \!\! \! \!\small{\xymatrix{
			0\ar[r]&\Z_2\{j_{p_4}\Sigma \nu'\eta_7\} \!\oplus\!\Z_2\{j_{p_4}\nu_4\eta_7\}  \ar[r]& \pi_{8}(M_{2^r}^4) \ar[r]&\Z_{2^{min\{3,r\}}}\{2^{u_r}\nu_5\}\ar[r]& 0\\
			0\ar[r]&\Z_2\{j_{p_4}\nu_4\eta_7\}\ar[u]\ar[r]& \pi_{8}(M_{2}^4)\ar[u]_{c_{5\ast}}\ar[r] &\Z_2\{4 \nu_5\} \ar@^{(->}[u]_{injective}\ar[r]& 0
	} }. \label{Diagram pi8(M2^1)to pi8(M2^r)}
\end{align}
Since $\pi_{8}(M_{2}^4)\cong \Z_2\oplus\Z_2$ \cite{WJProjplane} , Lemma 2.9 of \cite{Resultsfrom} implies 
\begin{align}
	\pi_{8}(M_{2^r}^4)\cong
	\Z_{2^{min\{3,r\}}}\oplus\Z_{2}\oplus \Z_2, r\geq 2. \label{pi8(Mr4)}
\end{align}
we complete the proof of  (\ref{pi8(Mr4)}).

\textbullet \quad For $M_{2^r}^5$, it easy to see $[\iota_5, 2^r\iota_5]=0$. By Lemma \ref{lem:exact m1+m2-1}, we have the following exact sequence 
\begin{align*}
		\pi_{9}(S^5)\xrightarrow{(2^{r}\iota_5)_{\ast}}\pi_{9}(S^5) \xrightarrow{i_{5\ast}}\pi_{9}(M_{2^r}^5) \xrightarrow{p_{5\ast}}\pi_{9}(S^{6})\xrightarrow{(2^{r}\iota_6)_{\ast}}\pi_{9}(S^{6})
\end{align*}

 we have short exact sequence with commutative square
\begin{align} \small{\xymatrix{
			0\ar[r]&\Z_2\{\nu_5\eta_8\}\ar[r]&\pi_{9}(M_{2^r}^{5})\ar[r]^-{p_{5\ast}}&\Z_{2^{min\{3,r\}}}\{2^{u_r}\nu_5\}\ar[r]& 0}}  \label{exact pi9(M^5)}
\end{align}
Clearly, $ \Z_{2^{min\{3,r\}}}\{2^{u_r}\nu_4\}\!\!=\!\!Ker(\partial_{7\ast}\!\!:\!\pi_{8}(S^5)\!\rightarrow\! \pi_{7}(F_{p_4}))\xrightarrow{\Sigma} Ker(\partial_{8\ast}\!\!: \!\pi_{9}(S^6)\!\rightarrow\! \pi_{8}(F_{p_5}))\!\!=\!\!\Z_{2^{min\{3,r\}}}\{2^{u_r}\nu_5\}$   is an isomorphism, by Lemma \ref{lem: Suspen Cf},  the short exact sequence (\ref{exact pi9(M^5)})  is split. Hence 
$\pi_{9}(M_{2^r}^5)\cong \Z_{2^{min\{3,r\}}}\oplus \Z_2$.
\end{proof}

\textbullet \quad For stable homotopy groups $\pi_{n+4}(M_{2^r}^n) (n\geq 6)$, we have the follwoing exact sequence 
\begin{align*}
	\pi_{n+4}(S^n)\xrightarrow{(2^{r}\iota_n)_{\ast}}\pi_{n+4}(S^n) \xrightarrow{i_{n\ast}}\pi_{9}(M_{2^r}^n) \xrightarrow{p_{n\ast}}\pi_{n+4}(S^{n+1})\xrightarrow{(2^{r}\iota_{n+1})_{\ast}}\pi_{n+4}(S^{n+1}),
\end{align*}
It is easy to get $\pi_{n+4}(M_{2^r}^n)\cong \Z_{2^{min\{3,r\}}} (n\geq 6)$.

\section{Elementary Chang-complexes }
\label{sec: Chang-complexes}
In this section we calculate homotopy groups of elementary Chang complexes appeared in the Theorem \ref{MainThm}.

In order to calculate homotopy groups in a unified and efficient way, in this section, by the rule $2^{\infty}=0$, one of the numbers $r,s$ in the space $C_{r}^{n+2,s}$,  which is a mapping cone of $f_{r}^{n+2,s}$, is allowed to be $\infty$. 
Since for $\infty\geq r,s\geq 1$, there is a  cofibration sequence
\begin{align}
	S^{n+1}\vee S^{n}\xrightarrow{f_{r}^{n+2,s}} S^{n+1}\vee S^{n}\xrightarrow{\lambda_{r}^{n+2,s}} \Cse \xrightarrow{q_{r}^{n+2,s}}  S^{n+2}\vee S^{n+1}
	\label{cofibre Cr^s n+2}
\end{align}
where  $f_{r}^{n+2,s}=(2^sj^{n+1}_1+ j^n_2\eta_n, 2^rj^n_2)$;  $\lambda_{r}^{n+2,s}$ and  $q_{r}^{n+2,s}$ are canonical inculsion and projection repectively.  For one of $s$ and $r$ is $\infty$, we have  
\begin{align*}
  C_{r}^{n+2,\infty}=C_{r}^{n+2}\vee S^{n+1} (r\geq 1); \qquad C_{\infty}^{n+2,s}=C^{n+2,s}\vee S^{n+1} (s\geq 1).
\end{align*}
Moreover,
 $f_{r}^{n+2,s}j^{n+1}_1=2^sj^{n+1}_1+ j^n_2\eta_n$; $f_{r}^{n+2,s}j^n_2=2^rj^n_2$.

Let $\Omega(S^{n+2}\vee S^{n+1})\xrightarrow{\partial_{r}^{n+2,s}} \F{n+2}\rightarrow \Cse \xrightarrow{q_{r}^{n+2,s}} S^{n+2}\vee S^{n+1}$   be the homotopy fiber sequence, where  $\F{n+2}$ denotes the homotopy fiber $F_{q_{r}^{n+2,s}}\simeq J( M_{f_{r}^{n+2,s}},  S^{n+1}\vee S^n)$.

By Lemma \ref{lemma Gray},
\begin{align}
	& Sk_{2n-1}(\F{n+2})\simeq J_1( M_{f_{r}^{n+2,s}}S^{n+1}\vee S^n) \simeq S^{n+1}
\vee S^n  \label{eq:J1Fn+2}\\
    & Sk_{3n-1}(\F{n+2})\simeq J_2( M_{f_{r}^{n+2,s}},  S^{n+1}\vee S^n) \nonumber\\
    &\simeq (S^{n+1}\vee S^n)\cup_{\gamma_{r}^{n+2,s}}C((S^{n+1}\vee S^n)\wedge (S^{n}\vee S^{n-1}))  \label{eq:J2Fn+2}
\end{align}
where $\gamma_{r}^{n+2,s}=[id_{S^{n+1}\vee S^n}, f_{r}^{n+2,s}]$.

Since $\gamma_{r}^{n+2,s}(j_2^{n}\wedge j_{2}^{n-1})=[id_{S^{n+1}\vee S^n}, f_{r}^{n+2,s}]\Sigma (j_2^{n-1}\wedge j_{2}^{n-1})\simeq [j_{2}^{n}, f_{r}^{n+2,s}j_{2}^{n}]$ $\simeq [j_{2}^{n}, 2^rj_{2}^{n}]\simeq 2^rj_{2}^{n}[\iota_{n}, \iota_{n}]$, we have 
\begin{align}
	Sk_{2n}(\F{n+2})\simeq (S^{n+1}\vee S^n)\cup_{2^rj_{2}^{n}[\iota_{n}, \iota_{n}]}CS^{2n-1}. \label{eq:Sk2nFn+2}
\end{align}
Similarly, by  $\gamma_{r}^{n+2,s}(j_1^{n+1}\wedge j_{2}^{n-1})=2^r[j_1^{n+1},j_2^n]$,  $\gamma_{r}^{n+2,s}(j_2^{n}\wedge j_{1}^{n})=2^s[j_1^{n+1},j_2^n]+ j_2^{n}[
\iota_n,\iota_n]\eta_{2n-1}$, we get 
\begin{align}
	Sk_{2n+1}(\F{n+2})\simeq (S^{n+1}\vee S^n)\cup_{\gamma_{2n+1,r}^{n+2,s}}C(S^{2n-1}\vee S^{2n}\vee S^{2n})  \label{eq:Sk2n+1Fn+2}
\end{align}
where the attaching map $\gamma_{2n+1,r}^{n+2,s}=(2^rj_{2}^{n}[\iota_{n}, \iota_{n}],~2^r[j_1^{n+1},j_2^n],~ 2^s[j_1^{n+1},j_2^n]\!+\!j_2^{n}[
\iota_n,\iota_n]\eta_{2n-1} )$.

Denote  the  $j^{n+2,s}_{r}$ the canonical  inclusion 
$j_{q_{r}^{n+2,s}}: S^{n+1}\vee S^n \hookrightarrow Sk_{m}(\F{n+2}) $ $(m\geq 2n)$ or  $S^{n+1}\vee S^n \hookrightarrow \F{n+2}$  without ambiguous.

Here we list the homotopy groups of some wegdes of spheres which will be used in the following content.
\begin{align}
	&\pi_{6}( S^5\vee S^4)=\Z_2\{j^5_1\eta_5\}\oplus\Z_2\{j^4_2\eta^2_4\};\nonumber\\
	&\pi_{7}( S^6\vee S^5)=\Z_2\{j^6_1\eta_6\}\oplus\Z_2\{j^5_2\eta^2_5\}; \nonumber\\
	&\pi_{7}( S^5\vee S^4)=\Z_2\{j^5_1\eta^2_5\}\!\oplus \!\Z_4\{j^4_2\Sigma\nu'\}  \oplus \Z_{(2)}\{j^4_2\nu_4\}; \nonumber\\
	&\pi_{8}( S^5\vee S^4)=\Z_8\{  j^5_1\nu_5\}\oplus \Z_2\{ j^4_2\nu_4\eta_{7}\}\oplus \Z_2\{ j^4_2\Sigma\nu'\eta_{7}\}\oplus \Z_{(2)}\{ [j_1,j_2]\};\nonumber\\
	&\pi_{8}( S^6\vee S^5)=\Z_2\{  j^6_1\eta^2_6\}\oplus \Z_8\{ j^5_2\nu_5\};\nonumber\\
	&\pi_{8}( S^7\vee S^6)=\Z_2\{j^7_1\eta_7\}\oplus\Z_2\{j^6_2\eta^2_6\}; \nonumber\\
	&\pi_{9}( S^6\vee S^5)=\Z_8\{j^6_1\nu_6\}\!\oplus \!\Z_2\{j^5_2\nu_5\eta_{8}\};\nonumber\\
	&\pi_{9}( S^7\vee S^6)=\Z_2\{j^7_1\eta^2_7\}\!\oplus \!\Z_8\{j^6_2\nu_6\}.\nonumber
\end{align}

\subsection{The $n+3$- homotopy groups of elementary Chang-complexes}
\label{n+3 Chang}
\qquad

\qquad

	$\pi_{7}(C_{\eta}^6)\cong \Z_2\oplus \Z_{(2)}$ and $\pi_{n+3}(C_{\eta}^{n+2})\cong \Z_4  (n\geq 5)$ come from Proposition  8.2 of  \cite{Mukai:pi(CPn)}.
\begin{lemma} \label{lem:pi7C6} Let $r,s\geq 1$
	\begin{align*}
	 &	\pi_{7}(C_{r}^{6})\cong
	\Z_2\oplus (1-\epsilon_r)\Z_{2}\oplus \Z_{2^{r+1}}\\
	&\pi_{7}(C^{6,s})\cong\Z_2\oplus \Z_{(2)}\oplus\Z_{2}\\
	& 	\pi_{7}(C_{r}^{6,s})\cong \Z_{2}\oplus\Z_{2}\oplus (1-\epsilon_r)\Z_{2}\oplus \Z_{2^{r+1}}.
	\end{align*}
\end{lemma}
\begin{proof}
	For $n=4$, 	by (\ref{eq:Sk2nFn+2}) and $[\iota_4,\iota_4]=\pm (2\nu_4-\Sigma \nu')$, 
	
	\begin{align}
	&\pi_7(\F{6})\!=\!\frac{\Z_2\{j_{r}^{6,s}j^5_1\eta^2_5\}\!\oplus\!\Z_{(2)}\{j_{r}^{6,s}j^4_2\nu_4\}\!\oplus\!\Z_4\{j_{r}^{6,s}j^4_2\Sigma\nu'\}}{\left \langle 2^{r+1}j_{r}^{6,s}j^4_2\nu_4-2^rj_{r}^{6,s}j^4_2\Sigma\nu' \right \rangle} ( \infty\!\geq\! r,s\geq 1)  \label{equ:pi7(Fr6s)} \\
	&=\left\{\!\!\!\begin{array}{ll}
		\Z_2\{j_{r}^{6,s}\!j^5_1\eta^2_5\}\oplus \Z_8\{j_{r}^{6,s}\!j^4_2\nu_4\}\oplus \Z_2\{2j_{r}^{6,s}\!j^4_2\nu_4\!-\!j_{r}^{6,s}\!j^4_2\Sigma\nu'\}, & \!\hbox{$r=1$;} \\
	\Z_2\{j_{r}^{6,s}j^5_1\eta^2_5\}\oplus \Z_{2^{r+1}}\{j_{r}^{6,s}\!j^4_2\nu_4\}\oplus \Z_4\{j_{r}^{6,s}j^4_2\Sigma\nu'\},& \hbox{$r\geq 2$.}\\ 
	\end{array}\right. \label{equ:pi7(Fr6s)case}
	\end{align}
	From Lemma \ref{lem: partial calculate1} and (\ref{exact Cf1}) of Lemma \ref{lem:exact m1+m2-1}, 
	we get the following exact sequence with a commutative square
	\begin{align} \footnotesize{\xymatrix{
				\!\!\!\!\pi_{8}( S^6\!\vee\! S^5)\ar[r]^-{ (\partial_{r}^{6,s})_{7\ast}}&\!\pi_{7}(\F{6})\ar[r]&\pi_{7}(\Ce{6})\ar[r]^-{q_{r\ast}^{6,s}}&\pi_{7}( S^6\!\vee\! S^5)\ar[r]^-{ f_{r \ast}^{7,s}}& \pi_{7}(S^6\!\vee\! S^5)\\
				\!\!\!\!\pi_{7}( S^5\!\vee\! S^4)\ar[r]^-{f^{6,s}_{r\ast}}\ar[u]_{\Sigma } &\pi_{7}(S^5\!\vee\! S^4)\ar[u]_{j^{6,s}_{r\ast} }&&} } \label{exact seq for pi7(Ce6)}
	\end{align}
Clearly, in above sequence 
\begin{align}
Ker f_{r\ast}^{7,s} =\Z_2\{j^5_2\eta^2_5\}  \label{Kerfr6s}
\end{align}

	By the commutative square in (\ref{exact seq for pi7(Ce6)})
	\begin{align}
		&(\partial_{r}^{6,s})_{7\ast}( j^6_1\eta^2_6)=(\partial_{r}^{6,s})_{7\ast}\Sigma  (j^5_1\eta^2_5)=j_{r\ast}^{6,s}f^{6,s}_{r\ast}(j^5_1\eta^2_5)=j_{r}^{6,s}(f^{6,s}_{r}j^5_1)\eta^2_5\nonumber \\
		&=j_{r}^{6,s}(j^5_1(2^s\iota_5)+ j^4_2\eta_4)\eta^2_5=j_{r}^{6,s}j^4_2\eta^3_4=2 j_{r}^{6,s}j^4_2\Sigma\nu'; \label{partial7 Z2}\\
		&(\partial_{r}^{6,s})_{7\ast}( j^5_2\nu_5)=(\partial_{r}^{6,s})_{7\ast}\Sigma (j^4_2\nu_4)=j_{r\ast}^{6,s}f^{6,s}_{r\ast}(j^4_2\nu_4)=j_{r}^{6,s}(f_{r}^{6,s}j^4_2)\nu_4\nonumber \\
		&=j_{r}^{6,s}(2^rj^4_2)\nu_4\!=\!j_{r}^{6,s}\!j^4_2(2^{2r}\nu_4\!-\!2^{r-1}(2^r\!-\!1)\Sigma\nu')~(\!\text{ Lemma A.1 of \cite{Resultsfrom}}~\!).\label{partial7 Z12}
	\end{align}
	From (\ref{equ:pi7(Fr6s)}), (\ref{partial7 Z2}) and (\ref{partial7 Z12}), we get $Coker (\partial_{r}^{6,s})_{7\ast}$ equals to 
	\begin{align*}
		&\frac{\Z_2\{j_{r}^{6,s}j^5_1\eta^2_5\}\!\oplus\!\Z_{(2)}\{j_{r}^{6,s}j^4_2\nu_4\}\!\oplus\!\Z_4\{j_{r}^{6,s}j^4_2\Sigma\nu'\}}{\left \langle  2 j_{r}^{6,s}j^4_2\Sigma\nu', 2^{r\!+\!1}j_{r}^{6,s}j^4_2\nu_4\!-\!2^rj_{r}^{6,s}j^4_2\Sigma\nu',2^{2r}j_{r}^{6,s}j^4_2\nu_4\!-\!2^{r\!-\!1}(2^r\!-\!1)j_{r}^{6,s}j^4_2\Sigma\nu'\right \rangle}  
	\end{align*}
	Let $a=j_{r}^{6,s}j^5_1\eta^2_5$, $b=j_{r}^{6,s}j^4_2\nu_4$, $c=j_{r}^{6,s}j^4_2\Sigma\nu'$, then 
		\begin{align*}
	Coker (\partial_{r}^{6,s})_{7\ast}&=\fra{\Z_2\{a\}\oplus\Z_{(2)}\{b\}\oplus\Z_4\{c\}}{2c, 2^{r\!+\!1}b-2^rc,2^{2r}b-2^{r\!-\!1}(2^r\!-\!1)c}	\\
	&=\Z_2\{a\}\oplus\frac{\Z\{b,c\}}{\left \langle 4c,2c, 2^{r\!+\!1}b-2^rc,2^{2r}b-2^{r\!-\!1}(2^r\!-\!1)c \right \rangle} \\
	&=\Z_2\{a\}\oplus\frac{\Z\{b,c\}}{\left \langle 2c, 2^{r\!+\!1}b, 2^{r\!-\!1}(2^r\!-\!1)c \right \rangle}\\
	&=\Z_2\{a\}\oplus \Z_{2^{r+1}}\{b\}\oplus (1-\epsilon_r)\Z_{2}\{c\}.
	\end{align*}
There is a map $\m{r}{4}\xrightarrow{\bar{\theta}}\Ce{6}$ making the following left ladder homotopy commutative and it induces the following right homotopy commutative ladder\\
\begin{align} \footnotesize{\xymatrix{
			S^4\ar[r]^-{2^r\iota_4}\ar[d]_{j^4_2} & S^4\ar[r]^{i_4}\ar[d]_{j^4_2}&  \m{r}{4}\ar[r]^-{p_4}\ar[d]^{\bar{\theta} }& S^5\ar[d]_{j^5_2}\\
			S^5\vee S^4 \ar[r]^-{f^{6,s}_{r}} &  S^5\vee S^4\ar[r]&\Ce{6} \ar[r]^-{q^{6,s}_{r}}&S^6\vee S^5,}}
	\footnotesize{\xymatrix{
			\Omega S^5\ar[r]^-{\partial}\ar[d]_{\Omega j^5_2} & F_{p_4}\ar[r]\ar[d]^{\theta}&  \m{r}{4}\ar[d]^{\bar{\theta} }\\
			\Omega (S^6\vee S^5) \ar[r]^-{\partial^{6,s}_{r}} &  \F{6}\ar[r]&\Ce{6}, }} \label{diagramMr4 to Ce6}
\end{align}
 The left homotopy commutative ladder in (\ref{diagramMr4 to Ce6}) induces the following commutative diagram:
	\begin{align}
		\! \!\! \! \!\small{\xymatrix{
				0\ar[r]&Coker\partial_{7\ast}\ar[d]\ar[r]& \pi_{7}(M_{2^{r}}^4)\ar[d]_{\bar{\theta}_{\ast}} \ar[r]&Ker\partial_{6\ast}=\Z_2\{\eta^2_5\} \ar[d]_{\cong}   \ar[r]& 0\\
				0\ar[r]&Coker (\partial^{6,s}_{r})_{7\ast}  \ar[r]&\pi_{7}(\Ce{6}) \ar[r] &Ker(\partial^{6,s}_{r})_{6\ast}\!=\!\Z_2\{j_2^5\eta^2_5\} \ar[r]& 0
		} } \label{Diagram pi7(M2^r4)to pi7(Ce6)}
	\end{align} 
	where the top short exact sequence comes from (3.15) of \cite{Resultsfrom}; $M_{2^{r}}^4\simeq S^4\vee S^5$ for $r=\infty$ and $Ker(\partial^{6,s}_{r})_{6\ast}\!=\!Ker f_{r\ast}^{7,s}\!=\Z_2\{j_2^5\eta^2_5\}$ in the bottom sequence.
	
	Since the top short exact sequence in (\ref{Diagram pi7(M2^r4)to pi7(Ce6)}) splits, so is the bottom one.
	
Hence, for $ \infty\!\geq\! r,s\geq 1$,	we have
	\begin{align}
		&\pi_{7}(C_{r}^{6,s})\!\cong \! Coker (\partial^{6,s}_{r})_{7\ast}\!\oplus\! Ker  f^{7,s}_{r\ast}\!\cong \!\Z_2\!\oplus\! \Z_{2^{r+1}}\!\oplus\! (1\!-\!\epsilon_r)\Z_2\oplus \Z_2.\nonumber
	\end{align}
	 $\pi_7(C_{r}^{6,\infty})=\pi_7(C_{r}^6\vee S^5)\cong  \pi_7(C_{r}^6)\oplus \pi_7(S^5)\cong \pi_7(C_{r}^6)\oplus \Z_2$ implies that 
	 \begin{align*}
\pi_{7}(C_{r}^{6})\cong
\Z_2\oplus (1-\epsilon_r)\Z_2\oplus\Z_{2^{r+1}}, r\geq 1.
	 \end{align*}
 $\pi_7(C_{\infty,1}^{6,s})=\pi_7(C^{6,s}\vee S^5)\cong \pi_7(C^{6,s})\oplus \Z_{2}  \cong \Z_{2}\oplus\Z_{(2)}\oplus \Z_{2}\oplus\Z_{2}$ implies that
 \begin{align*}
\pi_{7}(C^{6,s})\cong\Z_2\oplus \Z_{(2)}\oplus\Z_{2},s\geq 1.
 \end{align*}
We complete the proof of this lemma. 
\end{proof}

\begin{lemma} \label{lem:Stable pi(n+3)Chang} Let $r,s\geq 1$, $n\geq 5$
	\begin{align*}
		&	\pi_{n+3}(C_{r}^{n+2})\cong
	 \Z_{2}\oplus \Z_{2^{min\{2,r\}}}\\
		&\pi_{n+3}(C^{n+2,s})\cong\Z_2\oplus \Z_{4}\\
		& 	\pi_{n+3}(C_{r}^{n+2,s})\cong \Z_{2}\oplus\Z_{2}\oplus \Z_{2^{min\{2,r\}}}.
	\end{align*}
\end{lemma}
\begin{proof}
	These are stable homotopy groups in Lemma \ref{lem:Stable pi(n+3)Chang}, so without loss of generality, we prove the case $n=5$. We have the following exact sequence for $\infty\geq r,s\geq 1$
	\begin{align} \footnotesize{\xymatrix{
				\!\!\!\!\pi_{8}( S^{6}\!\vee\! S^5)\ar[r]^-{ f_{r\ast}^{7,s}}&\!\pi_{8}( S^{6}\!\vee\! S^5)\ar[r]&\pi_{8}(\Ce{7})\ar[r]^-{q_{r\ast}^{7,s}}&\pi_{8}( S^7\!\vee\! S^6)\ar[r]^-{ f_{r \ast}^{8,s}}& \pi_{8}(S^7\!\vee\! S^6)} } \label{exact seq for Stable Cen+2}
	\end{align}
Clearly, $Ker f_{r \ast}^{8,s}=\Z_2\{j_2^6\eta^2_6\}$. 
By 
\begin{align*}
	&f_{r \ast}^{7,s}(j_1^6\eta^2_{6})=(2^sj_1^6+j_2^5\eta_5)\eta^2_{6}=j_2^5\eta^3_5=4j_2^5\nu_5;
	&f_{r \ast}^{7,s}(j_2^5\nu_{5})=2^rj_2^5\nu_5
\end{align*}
we have $Coker f_{r \ast}^{7,s}=\Z_2\{j_1^6\eta^2_6\}\oplus \Z_{2^{min\{2,r\}}}\{j_2^5\nu_5\}$.
It is easy to see that the suspension homomorphism $Ker(\partial^{6,s}_{r})_{6\ast}\!\!=\! \!Ker f_{r\ast}^{7,s}\!\!=\! \!\Z_2\{j_2^5\eta^2_5\}\xrightarrow{\Sigma} Ker(\partial^{7,s}_{r})_{7\ast}\!\!=\! \!Ker f_{r\ast}^{8,s}\!\!=\! \!\Z_2\{j_2^6\eta^2_6\}$ is an isomorphism and 
 the bottom short exact sequence in (\ref{Diagram pi7(M2^r4)to pi7(Ce6)}) splits, so the following short exact sequence is also split by Lemma \ref{lem: Suspen Cf}
 \begin{align*}
 	0\rightarrow \Z_2\{j_1^6\eta^2_6\}\oplus \Z_{2^{min\{2,r\}}}\{j_2^5\nu_5\}\rightarrow \pi_{8}(\Ce{7})\rightarrow \Z_2\{j_2^6\eta^2_6\}\rightarrow 0.
 \end{align*}

	It shows 
	\begin{align*}
		\pi_{8}(C_{r}^{7,s})
		\cong \Z_{2}\oplus\Z_{2}\oplus \Z_{2^{min\{2,r\}}}.
	\end{align*}
	 $\pi_8(C_{r}^{7,\infty})=\pi_8(C_{r}^7\vee S^6)\cong  \pi_8(C_{r}^7)\oplus \pi_8(S^6)\cong \pi_8(C_{r}^7)\oplus \Z_2$ implies that 
	 \begin{align*}
\pi_{8}(C_{r}^{7})\cong
\Z_2\oplus\Z_{2^{min\{2,r\}}}, r\geq 1.
	 \end{align*}
	$\pi_8(C_{\infty}^{7,s})=\pi_8(C^{7,s}\vee S^6)\cong \pi_8(C^{7,s})\oplus \Z_{2} \cong \Z_2\oplus\Z_2\oplus\Z_{4}$ implies that
	\begin{align*}
\pi_{8}(C^{7,s})\cong\Z_2\oplus \Z_{4}, s\geq 1.
	\end{align*}
	We complete the proof of Lemma \ref{lem:Stable pi(n+3)Chang}.
\end{proof}

\subsection{The $n+4$-homotopy groups of elementary Chang-complexes}
\label{n+4 Chang}
\qquad

\qquad

In this section, we compute the  $n+4$ dimensional homotopy groups of elementary Chang-complexes in $\mathbf{A}_{n}^{2}$ for $n\geq 4$.
\begin{lemma} \label{lem:pi8C6} Let $r,s\geq 1$,
	\begin{align*}
		& \pi_{8}(C_{\eta}^6)\cong \Z_2\\
		& \pi_{8}(C_{r}^{6})\cong
		\Z_2\oplus \Z_{2^{min\{3,r+1\}}}\\
		&\pi_{8}(C^{6,s})\cong	\Z_2\oplus \Z_{2^{min\{3,s\}}}\oplus \Z_{2^{s+1}}\\
		& 	\pi_{8}(C_{r}^{6,s})\cong \Z_2\oplus \Z_{2^{min\{3,s\}}}\oplus \Z_{2^{min\{r,s+1\}}}\oplus \Z_{2^{min\{3,r+1\}}} .
	\end{align*}
\end{lemma}
\begin{clame}\label{clame:pi8F6} For $\infty \geq r,s\geq 1$,$\pi_{8}(\F{6})$ equals to 
  	\begin{align*} \small{
	\left\{\!\!
		\begin{array}{ll}
			\Z_8\{j^{6,s}_{r}\!j_1^5\nu_5\}\oplus \Z_2\{j^{6,s}_{r}\!j_2^4\nu_4\eta_7\}\oplus \Z_{2^{min\{r,s\!+\!1\}}}\{j^{6,s}_{r}[j_1^{5},j_2^4]\}, & \!\!\!\!\!\!\hbox{$\infty\neq  r\geq 1$;} \\
			\Z_8\{j^{6,s}_{r}\!j_1^5\nu_5\}\oplus \Z_2\{j^{6,s}_{r}\!j_2^4\nu_4\eta_7\}\oplus \Z_{2^{min\{r,s\!+\!1\}}}\{j^{6,s}_{r}[j_1^{5},j_2^4]\}\!\oplus\!\Z_{(2)}\{\tilde{j}_1^8\} , & \!\!\hbox{$r=\infty$ ,}
		\end{array}
		\right.}
	\end{align*}
	where $\tilde{j}_1^8\in \pi_{8}(\F{6})$  is a lift of $j_1^8\in 	Ker(\check{\partial}_{9,r}^{6,s})_{7\ast} $ in $(\ref{equ:Kerpartialrs69})$.

\end{clame}
\begin{proof}[Proof of Clame \ref{clame:pi8F6}]
	For $n=4$, by (\ref{eq:Sk2n+1Fn+2}), $Sk_{9}(\F{6})\simeq (S^5\vee S^4)\cup_{\gamma_{9,r}^{6,s}}C(S^7\vee S^8\vee S^8)$ with $\gamma_{9,r}^{6,s}=(2^rj_{2}^{4}[\iota_{4}, \iota_{4}],~2^r[j_1^{5},j_2^4],~ 2^s[j_1^{5},j_2^4]\!+\!j_2^{4}[
	\iota_4,\iota_4]\eta_{7} )=(2^{r+1}\!j_2^4\nu_4\!-\!2^rj_2^4\Sigma\nu',~2^r[j_1^{5},j_2^4],$ $~2^s[j_1^{5},j_2^4]\!+\!j_2^{4}\Sigma\nu'\eta_{7})$ (since the  sign``$\pm$" in $[\iota_4,\iota_4]=\pm(2\nu_4-\Sigma \nu')$ does not affect the homotopy type of $Sk_{9}((\F{6})$, here we chooes the positive sign).  There are the following  cofibration and  fibration sequences respectively
	\begin{align}
		&S^7\vee S^8\vee S^8\xrightarrow{\gamma_{9,r}^{6,s}}S^5\vee S^4\xrightarrow{j^{6,s}_{r}}Sk_{9}(\F{6})\xrightarrow{\check{q}_{9,r}^{6,s}}S^8\vee S^9\vee S^9;  \label{Cof:Sk9F}\\
		&\Omega(S^8\vee S^9\vee S^9)\xrightarrow{\check{\partial}_{9,r}^{6,s}} F_{\check{q}_{9,r}^{6,s}}\xrightarrow{}Sk_{9}(\F{6})\xrightarrow{\check{q}_{9,r}^{6,s}}S^8\vee S^9\vee S^9  \label{Fiber:Sk9F6}
	\end{align}
	where $Sk_{9}(F_{\check{q}_{9,r}^{6,s}})\simeq J_{1}(M_{\gamma_{9,r}^{6,s}}, S^7\vee S^8\vee S^8)\simeq S^5\vee S^4$.
	
	Denote the wedge of three spheres  $S^{n_1}\vee S^{n_2}\vee S^{n_3}$ by  $S^{\vee}_{(n_1,n_2,n_3)}$ and let $j^{n_i}_i$ be the canonical inclusion of the $i$-th wedge summand of $S^{n_1}\vee S^{n_2}\vee S^{n_3}$ for $i=1,2,3$.
	
	By Lemma \ref{lem: partial calculate1} and  Lemma \ref{stable exact seq}, we get the following exact sequence with a commutative square
	\begin{align} \small{\xymatrix{
				\!\!\!\!\!\!\pi_{8}(S^{\vee}_{(7,8,8)})\ar[r]^-{ (\gamma_{9,r}^{6,s})_{8\ast}}&\pi_{8}(S^5\!\vee\! S^4)\ar[r]^{j^{6,s}_{r\ast} }&\pi_{8}(\F{6})\ar[r]^-{\check{q}_{9,r\ast}^{6,s}}&\pi_{8}(S^{\vee}_{(8,9,9)})\ar[r]^-{ (\check{\partial}_{9,r}^{6,s})_{7\ast}}& \pi_{7}(F_{\check{q}_{9,r}^{6,s}})\\
				&&&\pi_{7}(S^{\vee}_{(7,8,8)})\ar[r]^-{(\gamma_{9,r}^{6,s})_{7\ast}}\ar[u]_{\cong\Sigma } &\pi_{7}(S^5\!\vee\! S^4)\ar@{_{(}->}[u]_{\cong }} } \label{exact seq for pi8F6}
	\end{align}
	where the wedge of three spheres  $S^{n_1}\vee S^{n_2}\vee S^{n_3}$ is denoted by  $S^{\vee}_{(n_1,n_2,n_3)}$.
	
	Clearly,  $(\gamma_{9,r}^{6,s})_{7\ast}:\pi_{7}(S^{\vee}_{(7,8,8)})\rightarrow \pi_{7}(S^5\!\vee\! S^4)$ is injective for $\infty \neq r\geq 1$ and $(\gamma_{9,\infty}^{6,s})_{7\ast}=0$, then  the commutative square in (\ref{exact seq for pi8F6}) gives 
	\begin{align}
		Ker(\check{\partial}_{9,r}^{6,s})_{7\ast}=\left\{
		\begin{array}{ll}
			0,  & \hbox{$\infty \neq r\geq 1$;} \\
			\Z_{(2)}\{j_{1}^8\}, & \hbox{$r=\infty$ .}
		\end{array}
		\right.  \label{equ:Kerpartialrs69}
	\end{align}
	 Moreover, $\pi_{8}(S^{\vee}_{(7,8,8)})=\Z_2\{j_1^7\eta_7\}\oplus\Z_{(2)}\{j_2^8\}\oplus\Z_{(2)}\{j_3^8\}$ and $(\gamma_{9,r}^{6,s})_{8\ast}(j_1^7\eta_7)=(2^{r+1}\!j_2^4\nu_4\!-\!2^rj_2^4\Sigma\nu')\eta_7=0$; $(\gamma_{9,r}^{6,s})_{8\ast}(j_2^8)= 2^r[j_1^{5},j_2^4]$; $(\gamma_{9,r}^{6,s})_{8\ast}(j_3^8)= 2^s[j_1^{5},j_2^4]\!+\!j_2^{4}\Sigma\nu'\eta_{7}$.  We get 
	\begin{align}
		&\pi_{8}(\F{6})\cong  Coker(\gamma_{9,r}^{6,s})_{8\ast} \nonumber\\
		=&\fra{\Z_8\{j_1^5\nu_5\}\oplus \Z_2\{j_2^4\nu_4\eta_7\}\oplus \Z_2\{j_2^4\Sigma\nu'\eta_7\}\oplus \Z_{(2)}\{[j_1^{5},j_2^4]\}}{2^r[j_1^{5},j_2^4],2^s[j_1^{5},j_2^4]\!+\!j_2^{4}\Sigma\nu'\eta_{7}}\nonumber\\ 
		=&\Z_8\{j_1^5\nu_5\}\oplus \Z_2\{j_2^4\nu_4\eta_7\}\oplus \fra{\Z_2\{j_2^4\Sigma\nu'\eta_7\}\oplus \Z_{(2)}\{[j_1^{5},j_2^4]\}}{2^r[j_1^{5},j_2^4],2^s[j_1^{5},j_2^4]\!+\!j_2^{4}\Sigma\nu'\eta_{7}}\nonumber\\ 
		=&\Z_8\{j_1^5\nu_5\}\oplus \Z_2\{j_2^4\nu_4\eta_7\}\oplus \Z_{2^{min\{r,s+1\}}}\{[j_1^{5},j_2^4]\}, \label{eqa:Cokergamma9r6s8}
	\end{align}
	since for $a=j_2^{4}\Sigma\nu'\eta_{7}$ and $b=[j_1^{5},j_2^4]$,
	\begin{align*}
		&\fra{\Z_2\{a\}\oplus \Z_{(2)}\{b\}}{2^rb,2^sb\!+\!a}=	\fra{\Z\{a,b\}}{2a,2^rb,2^sb\!+\!a}=\fra{\Z\{a,b\}}{2^{s+1}b,2^rb,2^sb\!+\!a}\\
		=&\fra{\Z\{2^sb\!+\!a,b\}}{2^{min\{r,s+1\}}b,2^sb\!+\!a}=\Z_{2^{min\{r,s+1\}}}\{b\}.
	\end{align*}
	Now we finish the proof of Claim \ref{clame:pi8F6} by (\ref{eqa:Cokergamma9r6s8}).
\end{proof}

\begin{proof}[Proof of Lemma \ref{lem:pi8C6}]
Firstly,  $\pi_{8}(C_{\eta}^6)\cong \Z_2$ comes from	Proposition  8.4 of  \cite{Mukai:pi(CPn)}.

For $\pi_{8}(\Ce{6}) (\infty\geq r,s\geq 1)$, we have the exact sequence with commutative square

\begin{align} \footnotesize{\xymatrix{
			\!\!\!\!\pi_{9}( S^6\vee S^5)\ar[r]^-{ (\partial^{6,s}_{r})_{8\ast}}&\pi_{8}(\F{6})\ar[r]&\pi_{8}(\Ce{6})\ar[r]^-{q^{6,s}_{r\ast}}&\pi_{8}( S^6\!\vee\! S^5)\ar[r]^-{ (\partial^{6,s}_{r})_{7\ast}}& \pi_{7}(\F{6})\\
			\!\!\!\!\pi_{8}( S^5\!\vee\! S^4)\ar[r]^-{f^{6,s}_{r\ast}}\ar[u]_{\Sigma } &\pi_{8}( S^5\vee S^4)\ar[u]_{j^{6,s}_{r\ast} }&& &} }. \label{exact seq for pi8C6}
\end{align}

From (\ref{equ:pi7(Fr6s)case}), (\ref{partial7 Z2}) and (\ref{partial7 Z12}), we get
\begin{align}
\Z_{2^{min\{3,r+1\}}}\cong Ker(\partial_{r}^{6,s})_{7\ast}=
	\left\{
	\begin{array}{ll}
		\Z_4\{j_1^6\eta^2_6+2j_2^5\nu_5\}, & \hbox{$r= 1$;} \\
		\Z_8\{j_1^6\eta^2_6+j_2^5\nu_5\}, & \hbox{$r=2$;}\\
		\Z_8\{j_2^5\nu_5\},& \hbox{$\infty\geq r\geq 3$.}\\
	\end{array}
	\right.  \label{Kerpartialr6s7}
\end{align}
By the commutative square in (\ref{exact seq for pi8C6})
\begin{align*}
		&(\partial^{6,s}_{r})_{8\ast}(j_1^6\nu_6)=j^{6,s}_{r}f^{6,s}_{r}(j_1^5\nu_5)=j^{6,s}_{r}(2^sj_1^5\!+\!j_2^4\eta_4 )\nu_5=2^sj^{6,s}_{r}\!j_1^5\nu_5\!+\!j^{6,s}_{r}\!j_2^4\Sigma\nu'\eta_7\\
		&(\partial^{6,s}_{r})_{8\ast}(j_2^5\nu_5\eta_8)=j^{6,s}_{r}f^{6,s}_{r}(j_2^4\nu_4\eta_7)=j^{6,s}_{r}(2^rj_2^4)\nu_4\eta_7=\begin{cases}	j^{6,s}_{r}j_2^4\Sigma\nu'\eta_7, r=1; \\
			0, \qquad \infty\geq r\geq 2. \quad \end{cases}
\end{align*}
By Clame \ref{clame:pi8F6}, $j^{6,s}_{r}j_2^4\Sigma\nu'\eta_7=0$ in $\pi_{8}(\F{6})$. Thus $Coker(\partial^{6,s}_{r})_{8\ast}$ equals 
\begin{align*} \small{
	\left\{\!\!\!\!
	\begin{array}{ll}
		\Z_{2^{min\{3,s\}}}\{j^{6,s}_{r}\!j_1^5\nu_5\}\!\oplus\! \Z_2\{j^{6,s}_{r}\!j_2^4\nu_4\eta_7\}\!\oplus\! \Z_{2^{min\{r,s\!+\!1\}}}\{j^{6,s}_{r}[j_1^{5},j_2^4]\}, & \!\!\!\!\!\!\!\!\!\!\hbox{$\infty\neq  r\geq 1$;} \\
		\Z_{2^{min\{3,s\}}}\{j^{6,s}_{r}\!j_1^5\nu_5\}\!\oplus\! \Z_2\{j^{6,s}_{r}\!j_2^4\nu_4\eta_7\}\!\oplus\! \Z_{2^{min\{r,s\!+\!1\}}}\{j^{6,s}_{r}[j_1^{5},j_2^4]\}\!\oplus\!\Z_{(2)}\{\tilde{j}_1^8\} , & \!\!\!\!\!\hbox{$r=\infty$ .}
	\end{array}
	\right.}
	\end{align*}
 For $s=1$ in $\Ce{6}$, the left homotopy commutative ladder in (\ref{diagramMr4 to Ce6}) induces the following commutative diagram 
\begin{align}
	\! \!\! \! \!\small{\xymatrix{
			0\ar[r]&Coker\partial_{8\ast}\ar[d]\ar[r]& \pi_{8}(M_{2^{r}}^4)\ar[d]_{\bar{\theta}_{\ast}} \ar[r]&Ker\partial_{7\ast}\cong\Z_{2^{min\{3,r\}}} \ar@{^{(}->}[d]^{j_{2\ast}^5|_{Ker\partial_{7\ast}}}   \ar[r]& 0\\
			0\ar[r]&Coker (\partial^{6,1}_{r})_{8\ast}  \ar[r]&\pi_{8}(C^{6,1}_{r}) \ar[r] &Ker(\partial^{6,1}_{r})_{7\ast}\!\cong\!\Z_{2^{min\{3,r+1\}}}\ar[r]& 0
	} } \label{Diagram pi8(M2^r4)to pi8(Ce16)}
\end{align} 
where the top exact sequence, which splits, comes from the top row of (\ref{Diagram pi8(M2^1)to pi8(M2^r)}) and the right map $j_{2\ast}^5|_{Ker\partial_{7\ast}}$ is injection. Moreover, the characteristic  
\begin{align*}
ch(Coker (\partial^{6,1}_{r})_{8\ast})=2^{min\{2,r\}} \leq 2^{min\{3,r\}}.
\end{align*}
Hence, by Lemma 2.9 of \cite{Resultsfrom}, the bottom short exact sequence in (\ref*{Diagram pi8(M2^r4)to pi8(Ce16)}) also splits.

For $s\geq 2$, we have the following left homotopy commutative ladder of the cofibration sequences which indues the right homotopy commutative ladder of the fibration sequences
	\begin{align} \footnotesize{\xymatrix{
				S^5\!\vee\!S^4\ar[r]^-{f_{r}^{6,1}}\ar@{=}[d]_{id} & S^5\!\vee\!S^4\ar[r]^{i_4}\ar[d]_{(2^{s\!-\!1}j_1^5, j_2^4)}&  C^{6,1}_{r}\ar[r]^-{q^{6,s}_{r}}\ar[d]^{\bar{\theta}_{s}^1 }& S^6\!\vee\!S^5\ar@{=}[d]_{id}\\
			S^5\!\vee\!S^4 \ar[r]^-{f^{6,s}_{r}} &  S^5\!\vee\!S^4 \ar[r]&\Ce{6} \ar[r]^-{q^{6,s}_{r}}&S^6\!\vee\!S^5,}}
		\footnotesize{\xymatrix{
				\Omega (S^6\vee S^5)\ar[r]^-{\partial^{6,1}_{r}}\ar@{=}[d]_{id} & F^{6,1}_{r}\ar[r]\ar[d]^{\theta_{s}^1}&  C^{6,1}_{r}\ar[d]^{\bar{\theta}_{s}^1 }\\
				\Omega (S^6\vee S^5) \ar[r]^-{\partial^{6,s}_{r}} &  \F{6}\ar[r]&\Ce{6}, }} \label{diagram Ce16 to Ce6}
	\end{align}
Then there is  the following commutative diagram 
\begin{align}
	\! \!\! \! \!\small{\xymatrix{
			0\ar[r]&Coker (\partial^{6,1}_{r})_{8\ast}  \ar[r]\ar[d]&\pi_{8}(C^{6,1}_{r}) \ar[d]^{\bar{\theta}_{s\ast}^1}\ar[r] &Ker(\partial^{6,1}_{r})_{7\ast}\!\cong\!\Z_{2^{min\{3,r+1\}}}\ar[r]\ar@{=}[d]& 0\\
			0\ar[r]&Coker (\partial^{6,s}_{r})_{8\ast}  \ar[r]&\pi_{8}(C^{6,s}_{r}) \ar[r] &Ker(\partial^{6,s}_{r})_{7\ast}\!\cong\!\Z_{2^{min\{3,r+1\}}}\ar[r]& 0
	} } \label{Diagram pi8(Ce16)to pi8(Ce6)}
\end{align} 
Since the top short exact sequence  in (\ref{Diagram pi8(Ce16)to pi8(Ce6)}) is split, so is the  bottom  one.

Hence, for $ \infty\!\geq\! r,s\geq 1$,	we have
\begin{align*}
\pi_{8}(C_{r}^{6,s})\cong &Coker (\partial_{r}^{6,s})_{8\ast}\!\oplus\! Ker (\partial_{r}^{6,s})_{7\ast}\\
\cong&	\left\{
	\begin{array}{ll}
	\Z_2\!\oplus\! \Z_{2^{min\{3,s\}}}\!\oplus\! \Z_{2^{min\{r,s\!+\!1\}}}\!\oplus\! \Z_{2^{min\{3,r\!+\!1\}}}  & \hbox{$\infty\neq r\geq 1$;} \\
		\Z_2\!\oplus\! \Z_{2^{min\{3,s\}}}\!\oplus\! \Z_{2^{min\{r,s\!+\!1\}}}\!\oplus\! \Z_{2^{min\{3,r\!+\!1\}}}\!\oplus\!\Z_{(2)} , & \hbox{$r=\infty$ .}
	\end{array}
	\right.
\end{align*}

$\pi_8(C_{r}^{6,\infty})=\pi_8(C_{r}^6\vee S^5)\cong  \pi_8(C_{r}^6)\oplus \pi_8(S^5)\oplus\pi_8(C_{r}^{10}) \cong \pi_8(C_{r}^6)\oplus \Z_8 \oplus \Z_{2^{r}}$ implies that 
\begin{align*}
	\pi_{8}(C_{r}^{6})\cong
	\Z_2\oplus \Z_{2^{min\{3,r+1\}}}, r\geq 1.
\end{align*}
$\pi_8(C_{\infty}^{6,s})=\pi_8(C^{6,s}\vee S^5)\cong \pi_8(C^{6,s})\oplus \pi_8(S^5) \oplus\pi_8(C^{10,s}) \cong \pi_8(C^{6,s})\oplus \Z_{8} \oplus \Z_{(2)}$, implies that
\begin{align*}
	\pi_{8}(C^{6,s})\cong	\Z_2\oplus \Z_{2^{min\{3,s\}}}\oplus \Z_{2^{s+1}},s\geq 1.
\end{align*}
We complete the proof of  Lemma \ref{lem:pi8C6}.  
\end{proof}

 \begin{lemma} \label{lem:pi9C7} Let $r,s\geq 1$,
 	\begin{align*}
 		& \pi_{9}(C_{\eta}^7)\cong \Z_2\\
 		& \pi_{9}(C_{r}^{7})\cong
 		\Z_2\oplus \Z_{2^{min\{3,r+1\}}}\\
 		&\pi_{9}(C^{7,s})\cong\Z_2\oplus\Z_{2^{min\{3,s\}}}\\
 		& 	\pi_{9}(C_{r}^{7,s})\cong \Z_2\oplus\Z_{2^{min\{3,s\}}}\oplus \Z_{2^{min\{3,r+1\}}}.
 	\end{align*}
 \end{lemma}
\begin{proof}
	For $C_{\eta}^7$, the cofibration sequence $S^6\xrightarrow{\eta_5}S^5\rightarrow C_{\eta}^7\rightarrow S^7\xrightarrow{\eta_6} S^6$ induces the following exact sequence by Lemma \ref{stable exact seq}
	\begin{align*}
		\pi_{9}(S^6)\xrightarrow{\eta_{5\ast}}\pi_{9}(S^5)\rightarrow   \pi_{9}(C_{\eta}^7)\rightarrow \pi_{9}(S^7) \xrightarrow{\eta_{6\ast}}\pi_{9}(S^6).
	\end{align*}
	It is easy to know $\pi_{9}(C_{\eta}^7)\cong Coker ~~\eta_{5\ast}\cong \pi_{9}(S^5)\cong \Z_2$ since $\eta_5\nu_6=0$.
	
	For $\Ce{7}$, by (\ref{eq:Sk2nFn+2}) and $[\iota_5,\iota_5]=\nu_5\eta_8$ (page 80 of \cite{Toda Whitehead Products}), we have 
	\begin{align*}
		Sk_{10}(\F{7})\simeq (S^6\vee S^5)\cup_{2^rj_2^5[\iota_5,\iota_5]} C S^9\simeq S^6\vee S^5\vee S^{10}.
	\end{align*}
	So  $j_{r}^{7,s}: S^6\vee S^5\hookrightarrow J_{2}(M_{f_{r}^{7,s}}, S^6\vee S^5)$ induces isomorphism $j_{r\ast}^{7,s}:\pi_{9}( S^6\vee S^5)\xrightarrow{\cong} \pi_9(J_{2}(M_{f_{r}^{7,s}}, S^6\vee S^5))$. By Lemma  (\ref{exact Cf}) of Lemma \ref{lem:exact m1+m2-1}, we get the  exact sequence for $\infty\geq r,s\geq 1$
\begin{align} \footnotesize{\xymatrix{
			\!\!\!\!\pi_{9}( S^{6}\!\vee\! S^5)\ar[r]^-{ f_{r\ast}^{7,s}}&\!\pi_{9}( S^{6}\!\vee\! S^5)\ar[r]&\pi_{9}(\Ce{7})\ar[r]^-{q_{r\ast}^{7,s}}&\pi_{9}( S^7\!\vee\! S^6)\ar[r]^-{ f_{r \ast}^{8,s}}& \pi_{9}(S^7\!\vee\! S^6)} } \label{exact seq for pi9 Ce7}
\end{align}
\begin{align*}
	&f_{r \ast}^{8,s}(j_1^7\eta^2_{7})=(2^sj_1^7+j_2^6\eta_6)\eta^2_{7}=j_2^6\eta^3_6=4j_2^6\nu_6;
	&f_{r \ast}^{8,s}(j_2^6\nu_{6})=2^rj_2^6\nu_6.
\end{align*}
In (\ref{exact seq for pi9 Ce7}) we have  
\begin{align}
	\Z_{2^{min\{3,r+1\}}}\cong Ker f_{r \ast}^{8,s}=
	\left\{
	\begin{array}{ll}
		\Z_4\{j_1^7\eta^2_7+2j_2^6\nu_6\}, & \hbox{$r= 1$;} \\
		\Z_8\{j_1^7\eta^2_7+j_2^6\nu_6\}, & \hbox{$r=2$;}\\
		\Z_8\{j_2^6\nu_6\},& \hbox{$\infty\geq r\geq 3$.}\\
	\end{array}
	\right.  \label{Kerfr8s}
\end{align}
\begin{align*}
	&f_{r \ast}^{7,s}(j_1^6\nu_{6})=(2^sj_1^6+j_2^5\eta_5)\nu_{6}=2^sj_1^6\nu_{6};
	&f_{r \ast}^{7,s}(j_2^5\nu_{5}\eta_8)=2^rj_2^5\nu_{5}\eta_8=0.
\end{align*}
In (\ref{exact seq for pi9 Ce7}) 
\begin{align*}
	Coker f_{r \ast}^{7,s}=\Z_{2^{min\{3,s\}}}\{j_1^6\nu_6\}\oplus \Z_{2}\{j_2^5\nu_5\eta_8\}.
\end{align*}

Comparing  (\ref{Kerpartialr6s7}) and (\ref{Kerfr8s}),  it is known that the  suspension homomorphism $Ker(\partial^{6,s}_{r})_{7\ast}\xrightarrow{\Sigma} Ker(\partial^{7,s}_{r})_{8\ast}\!\!=\! \!Ker f_{r\ast}^{8,s}$ is an isomorphism. Thus the following short exact sequence is  split by Lemma \ref{lem: Suspen Cf}
\begin{align*}
	0\rightarrow \Z_{2^{min\{3,s\}}}\oplus \Z_{2}\rightarrow \pi_{9}(\Ce{7})\rightarrow \Z_{2^{min\{3,r+1\}}}\rightarrow 0,
\end{align*}
	i.e.,  $\pi_{9}(\Ce{7})\cong \Z_2\oplus\Z_{2^{min\{3,s\}}}\oplus \Z_{2^{min\{3,r+1\}}} (\infty\geq r,s \geq 1)$.
	
	$\pi_9(C_{r}^{7,\infty})=\pi_9(C_{r}^7\vee S^6)\cong  \pi_9(C_{r}^7)\oplus \pi_9(S^6)\cong \pi_9(C_{r}^7)\oplus \Z_8$ implies that 
	\begin{align*}
		\pi_{9}(C_{r}^{7})\cong
		\Z_2\oplus \Z_{2^{min\{3,r+1\}}}, r\geq 1.
	\end{align*}
	$\pi_9(C_{\infty}^{7,s})=\pi_9(C^{7,s}\vee S^6)\cong \pi_9(C^{7,s})\oplus\pi_9(S^6)\cong \pi_9(C^{7,s})\oplus \Z_8$ implies that
	\begin{align*}
		\pi_{9}(C^{7,s})\cong\Z_2\oplus\Z_{2^{min\{3,s\}}},s\geq 1.
	\end{align*}
	We complete the proof of Lemma \ref{lem:pi9C7}.
\end{proof}

 \begin{lemma} \label{lem:Stable pin+4Cn+2} Let $r,s\geq 1$, $n\geq 6$
	\begin{align*}
		& \pi_{n+4}(C_{\eta}^{n+2})\cong 0\\
		& \pi_{n+4}(C_{r}^{n+2})\cong
		\Z_{2^{min\{3,r+1\}}}\\
		&\pi_{n+4}(C^{n+2,s})\cong\Z_{2^{min\{3,s\}}}\\
		& 	\pi_{n+4}(C_{r}^{n+2,s})\cong \Z_{2^{min\{3,s\}}}\oplus \Z_{2^{min\{3,r+1\}}}.
	\end{align*}
\end{lemma}
\begin{proof}
 Since the homotopy groups in this Lemma are stable, we only consider the case $n=6$ and there is  the following  exact sequence 
 \begin{align*} \footnotesize{\xymatrix{
 			\!\!\!\!\pi_{10}( S^{7}\!\vee\! S^6)\ar[r]^-{ f_{r\ast}^{8,s}}&\!\pi_{10}( S^{7}\!\vee\! S^6)\ar[r]&\pi_{10}(\Ce{8})\ar[r]^-{q_{r\ast}^{8,s}}&\pi_{10}( S^8\!\vee\! S^7)\ar[r]^-{ f_{r \ast}^{9,s}}& \pi_{10}(S^8\!\vee\! S^7).} }
 \end{align*}
 Then the proof of this lemma is the same  as that of Lemma \ref{lem:pi9C7} and we omit it.
\end{proof}

Although the generators of the homotopy groups in Theorem \ref{MainThm} are not listed, we can write the generators of these homotopy groups easily by the  proof of the lemmas in Section \ref{sec: Moore space} and Section \ref{sec: Chang-complexes}.
\newline
~~

\textbf{Example.}~~The generators of $\pi_8(\Ce{6})$ are given by the following
\begin{align*}	\Z_{2^{min\{3,s\}}}\{j_C^5\nu_5\}\!\oplus\! \Z_2\{j_C^4\nu_4\eta_7\}\!\oplus\! \Z_{2^{min\{r,s\!+\!1\}}}\{\lambda^{6,s}_{r}[j_1^{5},j_2^4]\}\oplus \Z_{2^{min\{3,r\!+\!1\}}}\{\tilde{\alpha}_{r}\}
\end{align*}
where $\lambda^{6,s}_{r}$ is given in (\ref{cofibre Cr^s n+2}) for $n=4$, which is also the composition  $S^5\vee S^4\stackrel{j_{r}^{6,s}}\hookrightarrow \F{6}\rightarrow \Ce{6}$ as pointed out in  the proof of Lemma 4.1.of \cite{Gray};
$j_{C}^5$ and $j_{C}^4$ are the compositions of canonical inclusions 
$S^5\stackrel{j_1^5}\hookrightarrow S^5\vee S^4 \stackrel{\lambda^{6,s}_{r}}\hookrightarrow \Ce{6}$  and 
$S^4\stackrel{j_2^4}\hookrightarrow S^5\vee S^4 \stackrel{\lambda^{6,s}_{r}}\hookrightarrow \Ce{6}$  respectively;
$\tilde{\alpha}_{r}$ is a lift of the generator of $Ker(\partial_{r}^{6,s})_{7\ast}$ in (\ref{Kerpartialr6s7}).

\subsection*{Acknowledgements}

The second author was supported by National Natural Science Foundation of China (Grant No. 11701430).

\end{document}